\documentclass[a4paper,reqno,oneside,10pt]{amsart}
\usepackage[utf8]{inputenc}
\usepackage{amsfonts,amsmath,amsthm,eucal,url,cite}
\usepackage{enumitem}
\usepackage{tikz-cd}
\usepackage{hyperref}
\usepackage{multicol}
\usepackage{slashed}
\usepackage{amssymb} 

\allowdisplaybreaks[4]

\renewcommand{\Tilde}{\widetilde}

\newcommand{\RR}{\mathbb{R}}
\newcommand{\CC}{\mathbb{C}}
\newcommand{\NN}{\mathbb{N}}

\newcommand{\ZZ}{\mathbb{Z}}

\def\Span{\mathrm{Span}}

\newtheorem{theorem}{Theorem}
\newtheorem{prop}[theorem]{Proposition}

\newtheorem{lemma}[theorem]{Lemma}
\newtheorem{definition}[theorem]{Definition}

\newtheorem{example}[theorem]{Example}

\theoremstyle{definition}

\theoremstyle{remark}
\newtheorem{rem}[theorem]{\bf Remark}

\title[Torsion-free connections on $G$-structures]{Torsion-free connections on $G$-structures}
\author[B. Flamencourt]{Brice Flamencourt}
\address[B. Flamencourt]{Universität Stuttgart, Institut für Geometrie und Topologie, Fachbereich Mathematik, Pfaffenwaldring 57, 70569 Stuttgart, Germany.}
\email{brice.flamencourt@mathematik.uni-stuttgart.de}


\begin{document}

\maketitle

\begin{abstract}
We prove that for a Lie group $\mathrm{SO}_n(\mathrm{R}) \subset G \subset \mathrm{GL}_n (\mathrm{R})$, any $G$-structure on a smooth manifold can be endowed with a torsion free connection which is locally the Levi-Civita connection of a Riemannian metric in a given conformal class. In this process, we classify the admissible groups.
\end{abstract}

\section{Introduction}

Let $M$ be a smooth manifold of dimension $n$ and $G$ a closed subgroup of $\mathrm{GL}_n (\RR)$. A $G$-structure on $M$ is a reduction of the frame bundle of $M$ to $G$ \cite[Chapter 4]{AM}.

We recall the following well-known result (see for example \cite{Crai}):

\begin{prop} \label{Gstructw}
Let $G$ be a closed subgroup of $\mathrm{GL}_n (\RR)$ containing $\mathrm{SO}_n (\RR)$ and let $P$ be a $G$-structure on $M$. Then, there exists a torsion-free connection on $P$. 
\end{prop}

Proposition~\ref{Gstructw} was stated as an exercise in \cite[Section 17.4, exercise (1)]{AM}. However, the author of this exercise believed that a more precise result occurs if we consider the following strategy of proof: we take a reduction of $P$ to $\mathrm{O}_n (\RR)$ or $\mathrm{SO}_n (\RR)$, which defines a Riemannian metric $g$, and then the Levi-Civita connection of $g$ is the desired torsion-free connection. This implies the stronger result that the connection on $P$ can be taken to be the Levi-Civita connection of a metric on $M$. However, such a reduction fails to exist in general, as shown by the following example:

\begin{example} \label{Counter-ex}
We consider the circle $S^1 \subset \CC$, parametrized by the map $\psi : [0, 2 \pi) \ni \theta \to e^{i \theta}$. Its tangent bundle is given by $TS^1 \simeq S^1 \times \RR$, and its frame bundle is $\mathrm{Fr} (S^1) \simeq S^1 \times \RR^*$. Let $G$ be the closed subgroup of $\RR^*$ generated by $2$, and let $P$ be the $G$-structure of $S^1$ given by $P_{\psi(\theta)} = \{ \psi(\theta) \} \times 2^\frac{\theta}{2 \pi} G$ for any $\theta \in [0, 2 \pi)$. There is no reduction of $P$ to $G \cap \mathrm{O}_1(\RR) = \{ 1 \}$ because $P$ is a non-trivial principal bundle. Note that the bundle $P$ is an embedding of the universal cover of $S^1$ into its frame bundle.
\end{example}

Nevertheless, we can prove that the torsion-free connection in the setting of Proposition~\ref{Gstructw} is locally induced by a Riemannian metric, and this is the object of this note. This can be described more precisely using conformal geometry. A discussion about the following definitions can be found in \cite[Section 2.2.2]{BH3M}. 

We recall that the conformal group $\mathrm{CO}_n (\RR)$ is the group of all matrices $\lambda S$ for $(\lambda, S) \in \RR^* \times \mathrm{O}_n (\RR)$. Given a Riemannian metric $g$ on $M$, the conformal class $[g]$ of $g$ is the set of all the metrics $g'$ such that there exists a function $f : M \to \RR$ satisfying $e^{2f} g' = g$. There is a one-to-one correspondence between $\mathrm{CO}_n(\RR)$-structures and conformal classes $[g]$, since any restriction of a $\mathrm{CO}_n(\RR)$-structure to $\mathrm{O}_n(\RR)$ defines a Riemannian metric $g$. We can then define the Weyl structures, which are the analogue in conformal geometry of the Levi-Civita connection in Riemannian geometry:

\begin{definition} \label{Weyl}
A Weyl connection (or structure) on a conformal manifold $(M, [g])$ is a torsion-free connection $\nabla$ such that there exists a $1$-form $\theta$, called the Lee form of $\nabla$ with respect to $g$, satisfying $\nabla g = -2 \theta \otimes g$.

Equivalently, if $P$ is the $\mathrm{CO}_n(\RR)$-structure associated to the conformal class $[g]$, a Weyl structure is a torsion-free connection on $P$.
\end{definition}

In Definition~\ref{Weyl}, when the Lee form $\theta$ is closed, -which does not depend on the choice of the metric in the conformal class,- then the Weyl structure $\nabla$ is called {\em closed}. This is equivalent to the fact that around any point in $M$, $\nabla$ is the Levi-Civita connection of a metric in $[g]$.

With this definition, we state the result we will prove in this note:

\begin{theorem} \label{maintheorem}
Let $M$ be a smooth $n$-dimensional manifold. Let $G$ be a closed subgroup of $\mathrm{GL}_n (\RR)$ containing $\mathrm{SO}_n (\RR)$ and let $P$ be a $G$-structure on $M$. Then, there is a reduction $Q$ of $P$ to $G \cap \mathrm{CO}_n (\RR)$ and a torsion-free connection on $Q$ such that the connection induced on the $\mathrm{CO}_n (\RR)$-structure given by the extension of $Q$ to $\mathrm{CO}_n (\RR)$ is a closed Weyl structure.
\end{theorem}

In order to illustrate this theorem and to give an idea of its proof, we come back to Example~\ref{Counter-ex}. Here, even if there is no reduction of the structure group to $\mathrm{SO}_1 (\RR) = \{1\}$, we can observe that $G$ seats inside $ \mathrm{CO}_1(\RR)$ and is a discrete group. Hence, the pull-back of the bundle $P$ to the universal cover $\RR$ of $S^1$ is a trivial bundle, whose total space is equal to $\{ (x, 2^{\frac{x}{2 \pi} + k}) \ \vert  \ x \in \RR, \ k \in \ZZ\}$, and the projection onto the basis is the first projection. The metric on $\RR$ given by $2^{-\frac{x}{\pi}} \langle \cdot, \cdot \rangle$ (where $\langle \cdot, \cdot \rangle$ stands for the standard metric) then induces a covariant derivative $\tilde \nabla$ on $\RR$. Seeing the smooth sections of $T \RR$ as smooth maps from $\RR$ to $\RR$, one has 
\[
\tilde \nabla_X Y = X \cdot \frac{d}{dx} Y - \frac{\ln 2}{2 \pi} X \cdot Y
\]
for any vector fields $X$, $Y$ on $\RR$. This connection descends to a connection $\nabla$ on $S^1$ because the translation $x \mapsto x + 2 \pi$ is an affine map. Moreover, $\nabla$ is torsion-free and compatible with $P$, so it is the connection given by Theorem~\ref{maintheorem} (which is here unique since $G$ is discrete).

We quickly outline the proof of Proposition~\ref{Gstructw}, using the analysis of \cite[Chapter 4]{Crai}. Denote by $\mathfrak{g}$ the Lie algebra of $G$ and by $\mathrm{ad} P$ the adjoint bundle of $P$ (which is a vector subbundle of the bundle of endomorphisms of $TM$). The set of connections on $TM$ compatible with $P$ is an affine space modeled on $\Omega^1 (M, \mathrm{ad} P)$. For any $\xi \in \Omega^1 (M, \mathrm{ad} P)$, we define $(\partial \xi) (X,Y) := \xi (X) (Y) - \xi (Y)(X)$ where $X, Y \in TM$ and we consider the set
\begin{equation}
\mathcal T_P := \frac{\Omega^2 (M, TM)}{\partial (\Omega^1 (M, \mathrm{ad} P))}.
\end{equation}
The {\em intrinsic torsion} $T_P^\mathrm{int}$ of $P$ is the equivalence class $[T_\nabla] \in \mathcal T_P$ where $T_\nabla$ is the torsion of any connection $\nabla$ compatible with $P$. This is well-defined because if $\nabla'$ is another connection, there is $\xi \in \Omega^1 (M, \mathrm{ad} P)$ such that $\nabla' = \nabla + \xi$, and an easy computation leads to $T_{\nabla'} = T_\nabla + \partial (\xi)$. Then, there exists a torsion-free connection on $P$ if and only if $T_P^\mathrm{int} = 0$.

For any $x \in M$, fix a frame $u \in P_x$ (which identifies $\RR^n$ with $T_x M$). For any $\phi \in \Lambda^2 (\RR^n)^* \otimes \RR^n$, let $\xi \in (\RR^n)^* \otimes \mathrm{End} (\RR^n)$ be given by
\begin{align}
2 \xi (X) (Y) := \phi (X,Y) - \phi(X, \cdot)^* (Y) - \phi (Y, \cdot)^* (X) && X, Y \in \RR^n,
\end{align}
where "$^*$" denotes the adjoint with respect to the standard metric on $\RR^n$. By construction, one has $\partial \xi = \phi$ and $\xi (X)$ is skew-symmetric for every $X \in \RR^n$. Since $\mathfrak{o}_n (\RR) \subset \mathfrak{g}$, we have $\xi \in (\RR^n)^* \otimes \mathfrak g$. We deduce that $\partial (\Omega^1 (M, \mathrm{ad} P)) = \Omega^2 (M, TM)$, implying $\mathcal T_P = 0$, thus $T_P^\mathrm{int} = 0$, which gives the result.

\section{Proof of Theorem~\ref{maintheorem}}

The proof of Theorem~\ref{maintheorem} relies on the classification of the subgroups of $\mathrm{GL}_n (\RR)$ containing $\mathrm{SO}_n (\RR)$.

In all this text, we will denote by $\mathrm{Diag} (a_1, \ldots, a_n)$ the diagonal matrix with diagonal coefficients $a_1, \ldots, a_n$ (we will also use this notation for any block diagonal matrix) and we will denote by $sgn :\RR \to \{ -1,0,1 \}$ the sign function. We first show the maximality of $\mathrm{SO}_n (\RR)$ in $\mathrm{SL}_n (\RR)$, which is a known result, but we recall a proof for the reader's convenance following partly an answer given by Yves Cornulier on the forum MathStackExchange.

\begin{lemma} \label{maximality}
Let $G$ be a subgroup of $\mathrm{SL}_n (\RR)$ containing $\mathrm{SO}_n (\RR)$. Then, $G = \mathrm{SL}_n (\RR)$ or $G = \mathrm{SO}_n (\RR)$.
\end{lemma}
\begin{proof} For $n = 1$ there is nothing to prove. For $n = 2$, suppose that there exists $A \in G \setminus \mathrm{SO}_2 (\RR)$. Using the singular value decomposition, one can assume that $A = \mathrm{Diag} (a, \frac{1}{a})$ with $a > 1$. For $\theta \in \RR$ let $R_\theta$ be the rotation of angle $\theta$. Let $\psi$ be the map which associates to an element of $\mathrm{SL}_n (\RR)$ the largest eigenvalue of the symmetric part of its polar decomposition. This map is continuous and  one has $\psi(A R_0 A) = a^2$ and $\psi(A R_{\pi/2} A) = 1$. Thus, by the intermediate value theorem, for any $x \in [1, a^2]$, the matrix $\mathrm{Diag} (x, \frac{1}{x})$ is in $G$, and this is true for any $k \in \NN$ and $x \in [1,a^k]$ by induction, so it is true for any $x > 1$, which gives the result.

It remains to treat the case where $n \ge 3$ using the case $n = 2$. Assume that there is $A \in G \setminus \mathrm{SO}_n (\RR)$. We can assume that $A$ is diagonal with positive coefficients using the singular value decomposition, thus $A = \mathrm{Diag} (a_1, a_2, \ldots, a_n)$, and conjugating by a suitable matrix in $\mathrm{SO}_n (\RR)$ we can assume that $a_1 \neq a_2$. Let $B \in \mathrm{SL}_n (\RR)$. We want to prove that $B \in G$. By another use of the singular value decomposition, we can assume that $B = \mathrm{Diag} (b_1, b_2, \ldots, b_n)$ where the coefficients are positive. Moreover, one has
\begin{align*}
\mathrm{Diag} (b_1, b_2, \ldots, b_n) = \mathrm{Diag} (b_1, \frac{1}{b_1}, 1, \ldots, 1) \cdot \mathrm{Diag} (1, b_1 b_2, \frac{1}{b_1 b_2}, 1, \ldots, 1) \\
\ldots \mathrm{Diag} (1, \ldots, 1, (b_1 \ldots b_{n-1}), \frac{1}{b_1 \ldots b_{n-1}} = b_n),
\end{align*}
thus it is sufficient to prove that each of the factors appearing in the right-hand side are in $G$. By conjugating by a suitable element of $\mathrm{SO}_n (\RR)$, it is sufficient to prove that for any $x > 0$ the matrix $\mathrm{Diag} (x, x^{-1}, 1, \ldots, 1)$ is in $G$. We now consider the matrix $R = \mathrm{Diag} (R_{\pi/2}, I_{n-2}) \in \mathrm{SO}_n (\RR)$, and we remark that
\[
A^{-1} R A R^{-1} = \mathrm{Diag} (\frac{a_2}{a_1}, \frac{a_1}{a_2}, 1, \ldots, 1).
\]
Since $\frac{a_2}{a_1} \neq 1$, the case $n = 2$ implies that $G$ contains all the matrices $\mathrm{Diag} (x, x^{-1}, 1, \ldots, 1)$ for $x > 0$, which gives that $B \in  G$ and concludes the proof.
\end{proof}

The following lemma will also be important in the description of the groups containing $\mathrm{SO}_n(\RR)$.

\begin{lemma} \label{lemma2}
Let $G$ be a subgroup of $\mathrm{GL}_n (\RR)$ containing $\mathrm{SO}_n (\RR)$, and let $x \in\det(G)$. Then, $\vert x \vert^{\frac{1}{n}} \mathrm{Diag}(sgn(x), 1, \ldots, 1) \in G$.
\end{lemma}
\begin{proof}
Let $x \in \det(G)$. There is a matrix $A \in G$ such that $\det(A) = x$, and using the singular value decomposition of $A$, there is a diagonal matrix $D \in G$ with $\det(D) = x$. If $D$ is of the form $\vert x \vert^{\frac{1}{n}} \mathrm{Diag} (\pm 1, \ldots, \pm 1)$ we have the conclusion of the lemma after multiplying by an element of $\mathrm{SO}_n (\RR)$ of the form $\mathrm{Diag} (\pm 1, \ldots, \pm 1)$, so we assume that $D^2 \notin \Span(I_n)$. There is a matrix $S \in \mathrm{SO}_n (\RR)$ with $S D^2 \neq D^2 S$. Let $B := D^{-1} S^T D S \in \mathrm{SL}_n (\RR)$. One has
\[
B B^T = D^{-1} S^T D S S^T D S D^{-1} = D^{-1} S^T D^2 S D^{-1} = (D^{-1} S D)^{-1} (D S D^{-1}),
\]
then
\[
B B^T = I_n \Leftrightarrow D^{-1} S D = D S D^{-1} \Leftrightarrow D^2 S = S D^2 ,
\]
and this last assertion is false, thus $B B^T \neq I_n$ and $B \notin \mathrm{SO}_n (\RR)$. By Lemma~\ref{maximality}, we conclude that $G \cap \mathrm{SL}_n (\RR) = \mathrm{SL}_n (\RR)$, and in particular $\vert x \vert^{\frac{1}{n}} \mathrm{Diag}(sgn(x), 1, \ldots, 1) D^{-1} \in G$, so $\vert x \vert^{\frac{1}{n}} \mathrm{Diag}(sgn(x), 1, \ldots, 1) \in G$ after multiplication by $D$ on the right.
\end{proof}

One writes $\mathrm{GL}_n (\RR) = \mathrm{SL}_n (\RR) \rtimes \RR^*$ with the identification $\{ \mathrm{Id} \} \rtimes \RR^* \to \mathrm{GL}_n (\RR), x \mapsto \vert x \vert^{\frac{1}{n}} \mathrm{Diag} (sgn(x), 1, \ldots, 1)$. We finally give the classification result:

\begin{prop} \label{structure}
Let $G$ be a subgroup of $\mathrm{GL}_n (\RR)$ containing $\mathrm{SO}_n (\RR)$. There exists a subgroup $H$ of $(\RR^*, \times)$ such that $G$ is equal to either $\mathrm{SO}_n (\RR) \rtimes H$ or $\mathrm{SL}_n (\RR) \rtimes H$. Moreover, if $G$ is closed, so is $H$.
\end{prop}
\begin{proof}
One has the following short exact sequence:
\begin{equation}
0 \to \mathrm{SL}_n (\RR) \cap G \to G \overset{\det}{\to} \det (G) \to 1.
\end{equation}
Now, let $\phi : H := \det (G) \to G$ given by $\phi (x) = \vert x \vert^{\frac{1}{n}} \mathrm{Diag}(sgn(x), 1, \ldots, 1)$, which is well-defined by Proposition~\ref{lemma2}. It is clear that $\phi$ is a morphism and $\det \circ \phi = id_H$, thus one has $G = (\mathrm{SL}_n (\RR) \cap G) \rtimes H$. Moreover, by Lemma~\ref{maximality} one has $\mathrm{SL}_n (\RR) \cap G = \mathrm{SL}_n (\RR)$ or $\mathrm{SL}_n (\RR) \cap G = \mathrm{SO}_n (\RR)$ because $G$ contains $\mathrm{SO}_n (\RR)$.

It remains to show that $H$ is closed when $G$ is closed. But if $H$ is non-discrete, $H \cap \RR_+^*$ has to be dense in $\RR_+^*$, so, $G$ being closed, it contains all the matrices of the form $\vert x \vert^{\frac{1}{n}} I_n$, $x \in \RR$, and then $H = \det G = \RR_+^*$ or $\RR^*$.
\end{proof}

\begin{rem}
Note that in Proposition~\ref{structure}, the semi-direct product is actually direct when $H \subset \RR_+^*$ or when $n$ is odd.
\end{rem}

Finally, we give the proof of the main theorem, for which we recall the following definition:

\begin{definition} \label{SIm}
Let $(N_1, g_1)$, $(N_2,g_2)$ be two Riemannian manifolds. A similarity between $N_1$ and $N_2$ is a diffeomorphism $\phi : N_1 \to N_2$ such that there exists $\lambda \in \RR^+_*$ with $\lambda^2 g_1 = \phi^* g_2$. In this case, $\lambda$ is called the ratio of the similarity.
\end{definition}

\begin{proof}[Proof of Theorem~\ref{maintheorem}]
According to Lemma~\ref{structure}, there is a closed subgroup $H$ of $\RR^*$ such that $G \simeq \mathrm{SO}_n (\RR) \rtimes H$ or $\mathrm{SL}_n (\RR) \rtimes H$. From the classification of the subgroups of $\RR^*$, $H$ is either $\RR^*$, $\RR^*_+$ or discrete.

{\bf First case: $H = \RR^*$ or $H = \RR^*_+$.} In this case, $G$ is either $\mathrm{GL}_n (\RR)$ or $\mathrm{CO}_n (\RR)$ or $\mathrm{GL}^+_n (\RR)$ or $\mathrm{CO}^+_n (\RR)$. In all these cases, there is a metric $g$ compatible with the $G$-structure, i.e. a reduction $P'$ of $P$ to $G \cap \mathrm{O}_n ({\RR})$. Then, the Levi-Civita connection of $g$ is torsion-free, so it induces a torsion-free connection on $P'$, and thus a torsion-free connection on the extension $Q$ of $P'$ to $G \cap \mathrm{CO}_n (\RR)$. The resulting connection on the extension of $Q$ to $\mathrm{CO}_n (\RR)$ is a closed (actually exact) Weyl structure because it is induced by the Levi-Civita connection of a metric on $M$.

{\bf Second case: $H$ is discrete.} Let $\Tilde M$ be the universal cover of $M$ and let $\Tilde P$ be the pull-back of $P$ to $\Tilde M$.

We first study the case $G = \mathrm{SO}_n (\RR) \rtimes H$. Then, the $H$-principal bundle $\Tilde P / \mathrm{SO}_n (\RR)$ is a covering of $\Tilde M$ so it is trivial. Every element $a \in H$ thus defines an $\mathrm{SO}_n (\RR)$-structure on $\Tilde M$ i.e. a metric $\Tilde g$. Since $\pi_1 (M)$ acts on $P / \mathrm{SO}_n (\RR)$ by multiplication by an element of $H$, we deduce that $\pi_1 (M)$ acts by similarities on $(\Tilde M, \Tilde g)$. Consequently, the Levi-Civita connection of $\Tilde g$ induces a torsion-free connection on $\Tilde P$ which descends to a torsion-free connection on $P$. We can take $Q := P$ in the statement of the theorem since $G \subset \mathrm{CO}_n (\RR)$. Finally, the resulting connection on the extension of $P$ to $\mathrm{CO}_n (\RR)$ is a closed Weyl structure because it is locally given by the Levi-Civita covariant derivative of a Riemannian metric defined by a local reduction of $P$ to $G \cap \mathrm{O}_n (\RR)$.

We consider now the case $G = \mathrm{SL}_n (\RR) \rtimes H$. Just as before, the $H$-principal bundle $\Tilde P / \mathrm{SL}_n (\RR)$ is trivial. Choosing an element $a \in H$ defines a $\mathrm{SL}_n (\RR)$-structure $\Tilde Q$ on $\Tilde M$ i.e. a volume form $\Tilde v$, and in particular an orientation on $\Tilde M$. Let $h$ be a Riemannian metric on $M$, and let $\Tilde h$ be its pull-back to $\Tilde M$. Let $v_h$ be the volume with respect to $\Tilde v$ of a $\Tilde h$-orthonormal frame of $T \Tilde M$ (note that $v_h^2$ does not depend on the choice of the frame). We define $\Tilde g := (v_h^2)^\frac{1}{n} \Tilde h$. Then, any oriented $\Tilde g$-orthonormal frame has volume $1$ with respect to $\Tilde v$. This implies that $\Tilde g$ defines a reduction of $\Tilde Q$ to $\mathrm{SO}_n (\RR)$. As in the previous case, $\pi_1 (M)$ acts on $P / \mathrm{SL}_n (\RR)$ by multiplication by an element of $H$, so for $\gamma \in \pi_1(M)$, $\gamma^* \Tilde v$ is a multiple of $\Tilde v$. Since, $\pi_1(M)$ acts by isometries on $(\Tilde M, \Tilde h)$, it acts by similarities on $(\Tilde M, \Tilde g)$. We finally conclude in the same way as for the case $G = \mathrm{SO}_n (\RR) \rtimes H$.
\end{proof}

From the proof we see that the principal bundle $Q$ defined in Theorem~\ref{maintheorem} has $\mathrm{SO}_n (\RR) \rtimes H'$ as structure group, where $H'$ is a discrete subgroup of $\RR_+^*$ (just take $H' := \{ 1 \}$ when $H = \RR^*$ or $\RR_+^*$, and $H' := H$ otherwise).

\vspace{1cm}

{\bf Acknowledgments.} The author thanks his advisor Andrei Moroianu for his help during the proofread of this note.

\renewcommand{\refname}

\end{document}